\documentclass[12pt, a4paper]{article}
\usepackage{titlesec} 
\usepackage{amsfonts}
\usepackage{amsthm}
\usepackage{amssymb}
\usepackage{amsmath}
\usepackage{theoremref}
\usepackage{enumerate}
\usepackage{bbm}
\usepackage{authblk}
\usepackage{mathtools}
\usepackage{xcolor}

\usepackage{xcolor}

\DeclarePairedDelimiter\abs{\lvert}{\rvert}%
\DeclarePairedDelimiter\norm{\lVert}{\rVert}%

\makeatletter
\let\oldabs\abs
\def\abs{\@ifstar{\oldabs}{\oldabs*}}
\let\oldnorm\norm
\def\norm{\@ifstar{\oldnorm}{\oldnorm*}}
\makeatother

\makeatletter
\g@addto@macro\bfseries{\boldmath}
\makeatother

\newcommand{\A}{\mathcal{A}}
\newcommand{\C}{\mathbb{C}}
\newcommand{\M}{\mathcal{M}}
\newcommand{\N}{\mathcal{N}}
\newcommand{\T}{\mathbb{T}}
\newcommand{\Ca}{\mathcal{C}}
\newcommand{\K}{\mathcal{K}}
\newcommand{\Po}{\mathcal{P}}
\newcommand{\E}{\mathcal{E}}

\newcommand{\z}{\zeta}

\newcommand{\conj}[1]{\overline{#1}}
\newcommand{\D}{\mathbb{D}}

\newcommand{\cD}{\conj{\mathbb{D}}}

\newcommand{\hil}{\mathcal{H}}

\newtheorem{thm}{Theorem}[section]
\newtheorem{lemma}[thm]{Lemma}
\newtheorem{cor}[thm]{Corollary}
\newtheorem{prop}[thm]{Proposition}
\theoremstyle{definition}

\theoremstyle{definition}

\begin{document}
\title{\textbf{An abstract approach to approximations in spaces of pseudocontinuable functions}}
\date{}
\author[1]{Adem Limani}
\author[2]{Bartosz Malman}
\affil[1]{Lund University, Lund, Sweden}
\affil[2]{KTH Royal Institute of Technology, Stockholm, Sweden}
\maketitle

\begin{abstract}
\noindent
We give an abstract approach to approximations with a wide range of regularity classes $X$ in spaces of pseudocontinuable functions $K^p_\vartheta$, where $\vartheta$ is an inner function and $p>0$. More precisely, we demonstrate a general principle, attributed to A. B. Aleksandrov, which asserts that if a certain linear manifold $X$ is dense in the space of pseudocontinuable functions $K^{p_0}_\vartheta$, for some $p_0>0$, then $X$ is in fact dense in $K^p_{\vartheta}$, for all $p>0$. 
Moreover, for a rich class of Banach spaces of analytic functions $X$, we describe the precise mechanism that determines when $X$ is dense in a certain space of pseudocontinuable functions. As a consequence, we obtain an extension of Aleksandrov's density theorem to the class of analytic functions with uniformly convergent Taylor series.
\end{abstract}

\section{Introduction.} 

Let $\D$ be the unit disc in the complex plane $\C$ and $H^p$ denote the classical Hardy spaces for $0<p<\infty$, which consists of analytic functions on $\D$ equipped with the (quasi)norm
\[
\| f\|_{H^p} := \sup_{0<r<1} \left( \int_{\T} |f(r\z)|^p dm(\z) \right)^{\min(1,1/p)} < \infty,
\]
where $m$ is the normalized arc length measure on the unit circle $\T$. We denote by $H^\infty$ the space of bounded analytic functions on $\D$ equipped with the supremum norm. As usual, we will identify functions in $H^p$ with their boundary functions belonging to a closed subspace of $L^p(\T, m)$. We shall assume the readers familiarity with the usual facts regarding the Hardy spaces $H^p$ on $\D$ and the basic facts about the boundary behaviors of functions in these space, for instance see \cite{garnett}. Functions $\vartheta$ belonging to $H^p$ and satisfying the property $|\vartheta(\z)| =1$ for $m$-a.e $\z \in \T$ are called \emph{inner functions} and enjoy the canonical factorization $\vartheta = BS_\mu$, where $B$ is the Blaschke product containing the zeros of $\vartheta$ on $\D$ and $S_\mu$ denotes the singular inner function associated to a finite positive singular Borel measure $\mu$, defined by
\begin{equation*} \label{SingIn} S_\mu(z) := \exp\Big(-\int_\T \frac{\zeta + z}{\zeta - z} d\mu(\z)\Big) \qquad z \in \D.
\end{equation*}
Let $\sigma(\vartheta)$ denote the boundary spectrum of an inner function $\vartheta$, the union of the support of the associated singular measure $\mu$ and the accumulation points of the zeros of $\vartheta$ contained in the Blaschke factor $B$. Then $\vartheta$ extends analytically across $\T \setminus \sigma(\vartheta)$ and its bad boundary behavior of $\vartheta$ is contained in $\sigma(\vartheta)$. Let $H^p_0$ denote the closed subspace of $H^p$ with $f(0)=0$ and for $0<p \leq \infty$, we define the spaces of pseudocontinuable functions $K_\vartheta^p$ as the set of $H^p$-functions $f$ satisfying the conditions

\begin{itemize}
\item[(i)][$\vartheta$-pseudocontinuable]

$f \in H^p \cap \vartheta \conj{H_0^p}$,

\item[(ii)][Analytic continuation]

$f$ has an analytic continuation across $\T \setminus \sigma(\vartheta)$.

\end{itemize}
We note that condition $(i)$ refers to the boundary function and the designation $\vartheta$-pseudocontinuable arises from the following fact. Any function $f(z) = \sum_{n\geq1}f_n z^n$ belonging to $H^p \cap \vartheta \conj{H^p_0}$ satisfies $f\conj{\vartheta} \in \conj{H^p_0}$ in the sense of boundary values, which gives rise to an extension $F_\vartheta (z) := \sum_{n\geq 1} f_n z^{-n}$ of $f\conj{\vartheta}$ to the exterior disk $\C_\infty \setminus \cD := \{z \in \C : |z| >1 \} \cup \{\infty\}$, which satisfies the properties that $F_\vartheta$ vanishes at $\infty$, $\lim_{r \to 1+} F_\vartheta(r\z) =\lim_{r\to 1-} f(r\z)/\vartheta(r\z) $, for $m$-a.e $\z \in \T$, and $F_\vartheta$ belongs to the Hardy space $H^p$ on $\C_\infty \setminus \cD$, that is 
\[
\sup_{r>1} \int_{\T} |F_\vartheta (r\z)|^p dm(\z) < \infty.
\]
See \cite{cima2000backward} for details. For $p\geq 1$, every $\vartheta$-pseudocontinuable function actually has an analytic continuation across $\T \setminus \sigma(\vartheta)$, thus item $(ii)$ is redundant in this case and one has $K_\vartheta^p := H^p \cap \vartheta \conj{H^p_0}$. However,  for $0<p<1$ this is no longer true, since for any $\z \in \T$, the rational functions $1/(1-\conj{\z}z)$ belong to $H^p \cap \vartheta \conj{H^p_0}$. We remark that it is not immediately clear that $(i)$ and $(ii)$ actually make $K^p_\vartheta$ into a closed subspace of $H^p$. For instance, this follows from Proposition 6.2.1 in \cite{cima2000backward}, a survey which treats the work of A. B. Aleksandrov in \cite{aleksandrov1979invariant} on backward shift invariant subspaces on $H^p$ for $0<p<1$.  

We denote by $X$ a linear manifold of analytic functions on $\D$ with the property that 
functions in $X$ enjoy some degree of regular extensions to the boundary $\T$, and we will loosely refer to them as regularity classes $X$. Our main purpose in this manuscript is to investigate, for a wide range of regularity classes $X$, satisfying some appropriate assumptions depending on the specific context, when the linear manifolds $X \cap K^p_\vartheta$ are dense in $K^p_\vartheta$. In the case $p=\infty$, this should be understood in the sense of the weak-star topology on $K^\infty_\vartheta$, inherited from $L^\infty(\T,m)$. The most well-known result within this category of approximation problems on spaces of pseudocontinuable functions is Aleksandrov's density theorem, which says that the disk algebra $\A$, consisting of analytic functions in $H^\infty$ with continuous boundary extensions to $\T$, is always dense in spaces of pseudocontinuable functions (see \cite{aleksandrovinv}). 
This result was later generalized to the setting of de Branges-Rovnyak spaces in \cite{comptesrenduscont}. For a wide range of regularity classes $X$, results for $K^2_\vartheta$ were recently treated in \cite{smoothdensektheta}. Other important works on topics relating regularity classes $X$ to spaces of pseudocontinuable functions has, for instance, appeared in \cite{starinvsmooth} by K. Dyakonov and D. Khavinson, in \cite{dyakonov1993division} by K. Dyakonov, and in \cite{aleksandrovinv} by A. B. Aleksandrov.  See also references therein.



\subsection{Main results.}

Let $X$ be a linear manifold  invariant under the backward shift and contained in $H^\infty$. The later condition together with Smirnov's maximum principle ensures that the intersection $\vartheta^*(X):= X \cap K^p_\vartheta$, is independent of $p>0$. Our first result treats a concept which originates back to the work of A. B.  Aleksandrov in \cite{aleksandrov1995existence}. There the author suggests that from the result that $\vartheta^*(\A)$ is dense in $K^2_\vartheta$, one can derive that $\vartheta^*(\A)$ is dense in $K^p_\vartheta$ for all $p\neq 2$, using results on invariant subspaces of $H^p$. Unfortunately, we have not been able to locate a proof or further elaborations on this matter in the literature. The main purpose of our first result is to demonstrate an abstract principle, which captures the essence of this particular idea. The principle in question essentially says that if a certain regularity class is dense in $K_\vartheta^{p_0}$, for some $p_0 >0$, then it is automatically dense in $K_\vartheta^p$, for all $p>0$. As indicated, this seems to be the intrinsic idea of A.B. Aleksandrov in \cite{aleksandrov1995existence}, hence with this perspective in mind, our efforts should be regarded as bringing these ideas into light and generalizing them.

\begin{thm} \thlabel{APthm} Let $\vartheta$ be an inner function and let $X$ be a linear manifold invariant under the backward shift and contained in $H^\infty$. If $\vartheta^*(X)$ is dense in $K^{p_0}_\vartheta$, for some $0<p_0 \leq \infty$, then it is dense in $K^p_\vartheta$, for any $0<p\leq \infty$.
\end{thm} 
\noindent
The proof of \thref{APthm} relies on characterizations of backward shift invariant subspaces of $H^p$, which are crucial tools that we summarize in the preliminary section. 
We shall now present an important corollary of \thref{APthm} on density of a wide range of regularity classes $X$.
Recall the following measure theoretical result (for instance, see Proposition 2.2 in \cite{smoothdensektheta}), which allows us to decompose any positive finite singular Borel measure $\mu$ on $\T$ into a unique sum of mutually singular measures 
\[
\mu = \mu_\Ca + \mu_\K,
\]
where $\mu_\Ca$ is concentrated on a countable union of Beurling-Carleson sets, that is, a countable union of closed sets $E \subset \T$ of Lebesgue measure zero, satisfying the condition
\[
\sum_{\nu} m(I_\nu) \log m(I_\nu) > -\infty,
\]
where $\{I_\nu\}$ are the connected components of $\T\setminus E$, and $\mu_\K$ vanishes on every Beurling-Carleson set. Consequently, every inner function $\vartheta$ factorizes as 
\[ \label{CAK}
\vartheta = B S_{\Ca} S_{\K},
\]
where $B$ is the Blaschke factor of $\vartheta$ and $S_\Ca$, $S_\K$ are singular inner functions associated to the singular measures $\mu_\Ca$, $\mu_\K$, respectively. Let $\A^\infty$ denote the space of bounded analytic functions on $\D$ with smooth extensions to the boundary $\T$ and let $\Lambda^{\alpha}_a$ denote the analytic H\"older classes, consisting of analytic functions on $\D$ with H\"older continuous extensions to $\T$ of order $0 < \alpha < 1$. According to the main result in \cite{smoothdensektheta}, we obtain the following immediate corollary of \thref{APthm}.
 
\begin{cor} \thlabel{smoothdense}
Let $0<p\leq \infty$ and $X$ be any backward shift invariant linear space with $\A^\infty \subseteq X \subseteq \Lambda^{\alpha}_a$, for some $0<\alpha < 1$. Then $\vartheta^*(X)$ is dense in $K^p_\vartheta$ if and only if $\vartheta = B S_{\Ca}$.
\end{cor}

We remark that the proof of Theorem 1.1 in \cite{smoothdensektheta} is partially based on a constructive method using Toeplitz operators and smoothing out singularities of singular inner functions concentrated on Beurling-Carleson sets. In fact, similar methods actually provide a proof of \thref{smoothdense} in the reflexive range $1< p <\infty$. However, Toeplitz operators are no longer useful outside this range, due to the simple fact they are no longer bounded on $H^p$, hence \thref{APthm} actually extends the main result in \cite{smoothdensektheta}.

In our next result, we shall investigate the precise mechanism that determines when a certain linear manifold $X \cap K^p_\vartheta$ is dense in $K^p_\vartheta$. With \thref{APthm} at hand, it suffices to determine this mechanism in the Hilbertian setting $K^2_\vartheta$. This time, we shall let $X$ be a Banach space of analytic functions on $\D$, which is contained in $H^2$ and satisfies some natural properties specified in the preliminary section, so that $X$ has a well-defined \textit{Cauchy dual} $X'$. That is, a Banach space $X'$ of analytic functions $g$ on $\D$, so that $X'$ becomes a dual space of $X$ with respect to the Cauchy-pairing 
\begin{equation}\label{C-dual}
\lim_{r \to 1-} \int_{\T} f(r\zeta) \conj{g(r\zeta)} dm(\zeta) \qquad f \in X, g \in X'.
\end{equation}
Our investigations will encompass a wide range of Banach spaces $X$, a couple of which we list below.
\begin{itemize}

\item The analytic continuously differentiable classes $\A^k$ with $k=0,1,2, \dots$, given by
\[
\left\{ f \in H^\infty: \norm{f}_{\A^k} := \sum_{j=0}^k \norm{f^{(j)}}_\infty < \infty \right\}.
\]
\item The analytic Sobolev spaces $\hil_a^s$ with $s > 0$, defined by
\[
\left\{ f(z) = \sum_{n=0}^\infty f_n z^n : \norm{f}^2_{\hil^s} := \sum_{n=0}^\infty (n+1)^s \abs{f_n}^2 < \infty \right\}.
\]
\item The space $U_a$ of analytic functions with uniformly convergent Taylor series on $\cD$, defined by

\[
\left\{ f(z) = \sum_{n=0}^\infty f_n z^n : \norm{f}_{U} := \sup_{N\geq 0} \,  \norm{\sum_{n=0}^{N} f_n z^n }_{\infty} < \infty \right\}.
\]

\end{itemize}
We remark that $\hil_a^1$ is the classical Dirichlet space, thus functions in $\hil_a^s$ enjoy weak regularity properties on $\T$ for $0 < s \leq 1$. However, if $s>1$ then $\hil_a^s$ contains the so-called Wiener algebra $W_a$, consisting of analytic functions on $\D$ with absolutely summable Taylor coefficients, and consequently it also contains $U_a$ and the analytic H\"older classes $\Lambda^\alpha_a$, with $\alpha > 1/2$. For even larger values of $s>1$, $\hil_a^s$ consists entirely of analytic functions with a certain number of continuous derivatives on $\T$. Denote by $[ \vartheta ]_{X'}$ the weak-star closure of analytic polynomial multiples of $\vartheta$ in $X'$. For a Banach space $X$ as above and an inner function $\vartheta$, we say that the pair $(X, \vartheta)$ satisfies the $(P)$-property if 
\[ 
[ \vartheta ]_{X'} \cap H^2 \subseteq \vartheta H^2.
\] 
In other words, the pair $(X, \vartheta)$ satisfies the $(P)$-property, if the inner factor $\vartheta$ is preserved under weak-star convergence in $X'$. It turns out that this property plays a decisive role in determining density of various regularity classes. 

\begin{thm} \thlabel{Fpropthm}
The pair $(X, \vartheta)$ satisfies the $(P)$-property if and only if $X\cap K^2_\vartheta$ is dense in $K^2_\vartheta$.  
\end{thm}

A few remarks are now in order. If $X$ is reflexive, then the weak and weak-star topologies on $X'$ agree and since $[\vartheta ]_{X'}$ is a convex set, we can rephrase the $(P)$-property of $(X, \vartheta)$, by replacing $[\vartheta]_{X'}$ with the norm-closure of analytic polynomial multiples of $\vartheta$ in $X'$. Furthermore, if the Cauchy dual $X'$ is equivalent to a Bergman space, which for instance happens for the class of analytic Sobolev spaces $X=\mathcal{H}_a^{s}$ with $s>0$, then $(X, \vartheta)$ satisfies the $(P)$-property if and only if $\vartheta = B S_\Ca$, where $B$ is a Blaschke product and $S_\Ca$ is a  singular inner function supported on a countable union of Beurling-Carleson sets (see \cite{roberts1985cyclic}). Connections between a wide range of regularity classes $X$ in $K_\vartheta$, having Bergman spaces as Cauchy duals, and between the theorem of Korenblum and Roberts (see \cite{korenblum1981cyclic},\cite{roberts1985cyclic}) on forward shift invariant subspaces generated by inner functions on Bergman spaces, has previously been employed in \cite{starinvsmooth} and in \cite{smoothdensektheta}. 

In light of \thref{Fpropthm}, it is natural to ask if there exists a practical sufficient condition on $X$, which ensures the $(P)$-property of $(X, \vartheta)$, for all $\vartheta$. This is the content of our next main result. 

\begin{thm} \thlabel{FpropHp}
Assume that $X'$ is continuously embedded in $H^p$, for some $p > 0$. Then $(X, \vartheta)$ satisfies the $(P)$-property, for all inner functions $\vartheta$.
\end{thm}
We shall now turn to an application of \thref{FpropHp}, with the intention of generalizing Aleksandrov's density theorem to a slightly finer class than the disk algebra $\A$. 
To this end, recall be the space $U_a$ of analytic functions in $\D$ with uniformly convergent Taylor series on $\cD$, that is, if $f(z) = \sum_{n=0}^\infty f_n z^n$ is the Taylor series of an analytic function $f$ with continuous restriction to $\T$, and $P_N f(z) = \sum_{0\leq n \leq N} f_n z^n$, then $U_a$ consists of those $f$ for which \[ \lim_{N \to \infty} \| P_N f - f\|_{\infty} = 0,\] the norm $\|\cdot \|_\infty$ being computed on the circle $\T$ (or, equivalently, over $\D$). We equip $U_a$ with the norm \[ \|f\|_U := \sup_{N \geq 0} \| P_Nf\|_\infty, \] which makes $U_a$ into a Banach space, strictly contained in $\A$, and for which polynomials are dense. 
A deep theorem of Vinogradov proved in \cite{vinogradov1976convergence} asserts that the Cauchy dual $U'_a$ is continuously embedded in $H^p$ for all $p \in (0,1)$. As a direct consequence of our abstract approach and Vinogradov's result, we obtain the following strengthening of Aleksandrov's famous density result from \cite{aleksandrovinv}.

\begin{cor} \thlabel{DenseinUa}
The linear manifold $\vartheta^*(U_a )$ is dense in $K_\vartheta^p$, for any inner function $\vartheta$ and 
 $0<p \leq \infty$.
\end{cor}

We want to mention that the corollary can be extended by replacing $K^2_\vartheta$ with a space from the class of Hilbert spaces of analytic functions studied in \cite{jfabackshift} on which the backward shift operator acts as a contraction. In particular, it applies with $K^2_\vartheta$ replaced by a de Branges-Rovnyak space. A proof of this, which is admittedly rather technical, can be deduced from the arguments in \cite{jfabackshift}, in particular by replacing our $(P)$-property and \thref{FpropHp} with Lemma 3.4 in \cite{jfabackshift}. 

We remark that one cannot strengthen \thref{DenseinUa} by replacing $U_a$ with the 
Wiener algebra $W_a$, which consists of analytic functions on $\D$ with absolutely summable Taylor coefficients and equppied with the $\ell^1$-norm. In fact, there exists an inner function $\vartheta$, such that $W_a \cap K^2_\vartheta = \{0\}$, see \cite{smoothdensektheta}. 
The problem of characterizing the inner functions $\vartheta$ for which $(W_a, \vartheta)$ satisfies the $(P)$-property seems quite difficult and the main obstacle is 
that the $\ell^1$-norm poorly reflect the precise boundary behavior of functions in $W_a$. For instance, see Kaufman in \cite{kaufman1979zero}, illustrating pathological behaviors of zero sets in $W_a$.

The manuscript is organized as follows. In the section of preliminary results, we collect the preparatory work which consists of justifying that the Cauchy duals of a wide range of Banach spaces $X$ are well defined, results on invariant subspaces of the backward shift on $H^p$, for all $p$. The last section is devoted to proofs of our main results, which include \thref{APthm}, \thref{Fpropthm} and \thref{FpropHp}. \thref{APthm} requires most of the work and relies on the results about invariant subspaces of the backward shift, presented in the preliminary section. Meanwhile, the proofs of \thref{Fpropthm} and \thref{FpropHp} are essentially self-contained.

\section{Preliminary results.}

\subsection{Cauchy duals of regularity classes $X$.}

In this section, we shall introduce a couple of natural assumptions on our regularity classes $X$ of considerations and briefly justify that their Cauchy duals are well-defined. Unless indicated otherwise, we shall denote by $X$ a Banach space of analytic functions on $\D$ with the following properties:

\begin{enumerate}
\item[(i)] The set of analytic polynomials $\Po$ are dense in $X$.

\item[(ii)] The norms of the monomials satisfy $\limsup_{n \to \infty} \norm{z^n}_{X}^{1/n} \leq 1$.

\item[(iii)] $X$ is continuously embedded in $H^2$. 
\end{enumerate}
\noindent
Let $X^*$ denote the Banach space dual of $X$ and consider the (conjugate) linear map $\mathcal{L}$ on $X^*$ into the space of analytic functions on $\D$ defined by 
\[
\mathcal{L}(\Lambda)(z) := \sum_{n=0}^{\infty} \conj{\Lambda(\zeta^n)} z^n \qquad  \Lambda \in X^* , \qquad z \in \D.
\] 
Property $(ii)$ ensures that $\mathcal{L}$ maps into the space of analytic functions on $\D$ and $(i)$ implies that $\mathcal{L}$ is injective. Now denote by $X'$ the image of $\mathcal{L}$ under $X^*$, equipped with the norm 
\[
\norm{\mathcal{L}(\Lambda)}_{X'} := \norm{\Lambda}_{X^*} \qquad \Lambda \in X^*, 
\]
so that $X'$ becomes a Banach space of analytic functions on $\D$ which is isometrically isomorphic to $X^*$. 
The space $X'$ is referred to as the \textit{Cauchy dual} of $X$, since for any $\Lambda \in X^*$ one has
\[
\Lambda (f) = \lim_{r \to 1-} \int_{\T} f(\zeta)\,  \overline{\mathcal{L}(\Lambda)(r\zeta)} dm(\zeta), \qquad f \in \Po.
\]
We remark that in all Banach spaces $X$ of our considerations contain the dilations $f_r(z):=f(rz)$ by $0<r<1$ of $f\in X$, and we actually have
$\norm{f-f_r}_X \to 0$ as $r\to 1-$. In particular, this justifies the Cauchy dual representation for all $f\in X$, stated in \eqref{C-dual}. The purpose of property $(iii)$ is to ensure that our regularity classes $X$ have sufficiently behaved boundary values, so that integral pairings with $K^2_\vartheta$-functions are well-defined and so that the Cauchy dual $X'$ contains $H^2$. Property $(i)$ is basically a convenient way to ensure that $X$ is a separable Banach space. We summarize this section by briefly identifying the Cauchy duals to some previously mentioned examples. The Cauchy dual of the disk algebra $\A$ is identified with Cauchy integrals of $M(\T)$, the space of finite complex Borel measures on $\T$. It is equipped with the norm
\[
\norm{f}_{\A'} := \inf \{ \norm{\mu} : f= P_+ (\mu), \, \mu \in M(\T) \},
\]
where $P_+ (\mu) (z) := \int_{\T} \frac{d\mu(\zeta)}{1-\conj{\zeta}z}$ denotes the Cauchy integral of $\mu$
and $\norm{\mu}$ denotes the total variation norm of $\mu$. The Cauchy duals of the spaces $\A^k$ with $k\geq 1$ essentially consists of Cauchy integrals of distributions up to order $k$ on $\T$, but with regards to functional analysis, spaces of continuously differentiable functions are more conveniently captured within the realm of Sobolev spaces. The analytic Sobolev spaces $\hil_a^s$ with $s > 0$ are Hilbert spaces of analytic functions $\D$, with Cauchy duals conveniently denoted by $\hil_a^{-s}$, due to the fact that they consist of the analytic functions on $\D$ having finite norm
\[
\norm{f}^2_{\hil^{-s}} := \sum_{n=0}^{\infty} (n+1)^{-s} \abs{f_n}^2. 
\]
One can show that the Cauchy dual of $\hil_a^{-s}$ is again $\hil_a^{s}$, thus making these spaces reflexive in the Cauchy dual pairing. Moreover, it is not difficult to prove that $\hil_a^{-s}$ is equivalent to the Bergman space $L^{2,s-1}_a$, given by
\[
\norm{f}^2_{L^{2,s-1}} := \int_{\D} \abs{f(z)}^2 (1-|z|)^{s-1} dA(z)
\]
where $dA$ denotes the area-measure on $\D$. 

\subsection{Backward shift invariant subspaces on $H^p$.}

Here we collect some results on invariant subspaces of the backward shift operator  
\[
f \mapsto \left(f-f(0) \right)/z
\]
on $H^p$. In the reflexive range $1<p<\infty$, Douglas, Shapiro and Shields showed in \cite{douglas1967cyclic} that any closed invariant subspace of the backward shift is a space of pseudocontinuable functions $K^p_\vartheta$, for some inner function $\vartheta$. The same results remains true for $p=1$, but requires a more subtle proof. For the sake of future reference, we state these results below and refer the reader to the survey in \cite{cima2000backward} for detailed treatments.
\begin{prop}\thlabel{Aleks1} 
Let $1\leq p < \infty$. Then $\M_p$ is a closed invariant subspace of the backward shift operator on $H^p$ if and only there exists an inner function $\vartheta$ such that $\M_p = K_\vartheta^p$.
\end{prop}
\noindent
We shall also need Aleksandrov's description of the backward shift invariant subspaces $\E$ of $H^p$ for $0<p <1$ in \cite{aleksandrov1979invariant}. These are way more complicated and cannot simply be captured by $K^p_\vartheta$. In fact, to give a complete description we need the following three parameters: an inner function $\vartheta$, a closed set $ \sigma(\vartheta) \subseteq F \subset \T$ determined by $F := \{ \zeta \in \T: (1-\conj{\z}z)^{-1} \in \E \}$, a integer-valued function $\kappa$ on $F$ determined by $\kappa(\z) = \max \{ n \in \mathbb{N}: (1-\conj{\z}z)^{-n} \in \E \}$. 

\begin{prop}[Aleksandrov, \cite{aleksandrov1979invariant} ]\thlabel{Aleks} 
Every closed invariant subspace of the backward shift operator on $H^p$ with $0<p<1$ is of the form $\E^p(\vartheta, F, \kappa)$, for some parameters $\theta, F, \kappa$ as above, and the space is defined by the following properties: 
\begin{itemize}
\item[(i)] 
$ f \in H^p \cap \vartheta \conj{H_0^p}$.
\item[(ii)]
$f$ has analytic continuation across $\T \setminus F$.
\item[(iii)]
$f$ has a pole of order no larger than $\kappa(\z)$ at $\z \in F_0 \setminus \sigma(\vartheta)$, where $F_0$ denotes the isolated points of $F$.

\end{itemize}
\end{prop}
%



The last piece of equipment in this section is a characterization of the backward shift invariant subspaces of $H^\infty$. To this end, we shall need a result on a Beurling-type theorem on the pre-dual of $H^\infty$, identified with respect to the Cauchy dual pairing with the Banach space $\K_a :=P_+(L^1)$, which consists of Cauchy integrals of absolutely continuous finite Borel measures on $\T$, regarded as a subspace of Cauchy integrals of finite complex Borel measures on $\T$.

\begin{prop} \thlabel{beurlingKa}
 Every weak-star closed backward shift invariant subspace of $H^\infty$ is of the form $K_\vartheta^\infty$, for some inner function $\vartheta$.
\end{prop}

\begin{proof} Note that since the forward shift is a continuous operator on $\K_a$, its Banach space adjoint, the backward shift, is a weak-star continuous operator on $H^\infty$. Hence if $\M$ is a weak-star closed backward shift invariant subspace of $H^\infty$, then the pre-annihiliator $\M_\perp$ of $\M$ is a closed forward shift invariant subspace of $\K_a$. By Aleksandrov's theorem \cite[Theorem 11.2.3]{cauchytransform}, there exists an inner function $\vartheta$, such that 
\[ 
\M_\perp = \vartheta(\K_a) :=  \{ f \in \K_a : f/\vartheta \in \K_a \}.
\]
It is straightforward to see that $\vartheta(\K_a)$ is the norm-closure of polynomial multiples of $\vartheta$ in $P_+(L^1)$, thus using the fact that $\M$ is weak-star closed, we obtain 
\[
\M = \left( \M_\perp \right)^\perp = \left( \vartheta (\K_a) \right)^\perp = K_\vartheta^\infty.
\]

\end{proof}


\section{Proofs of main results.}

In this section, we shall first establish \thref{APthm} by using results on backward shift invariant subspaces on $H^p$. For this reason, our proof naturally splits into different cases, depending on the range of $0<p \leq \infty$, so that each case allows for an application of the appropriate tool. 
In the proof below, we shall use the notion that an inner function $\vartheta$ divides another inner function $\phi$, if the quotient $\phi/ \vartheta$ again is an inner function.

\begin{proof}[Proof of \thref{APthm}]
Recall that since $X \subset H^\infty$, it follows by Smirnov's maximum principle that the linear manifolds $\vartheta^*(X) :=X \cap K^p_\vartheta$ are independent of $p$, thus for any $0<p \leq \infty$ we have
\begin{equation}\label{pinv}
\vartheta^*(X) \subset K^p_\vartheta.
\end{equation}
Throughout this proof, we shall denote by $\M_p$ the closure of $\vartheta^*(X)$ in $H^p$ for $0<p< \infty$ and let $\M_\infty$ denote the weak-star closure of $\vartheta^*(X)$ in $H^\infty$. 
Now suppose that $\vartheta^*(X)$ is dense in $K^{p_0}_\vartheta$, for some $0<p_0 \leq \infty$ and fix an arbitrary $1\leq p \leq \infty$. Since $X$ is assumed to be invariant under the backward shift, we have that $\M_p$ is a closed backward shift invariant subspace of $H^p$, contained in $K^p_\vartheta$. If $1\leq p< \infty$, then we can according to \thref{Aleks1} find an inner function $\phi$, such that $\M_p = K^p_\phi \subseteq K^p_\vartheta$. If $p=\infty$, then we instead apply \thref{beurlingKa} which yields $\M_\infty = K^\infty_\phi \subseteq K^\infty_\vartheta$.
Thus, in any case $1\leq p \leq \infty$, we conclude that $\phi$ divides $\vartheta$. Next, we combine these observations together with \eqref{pinv} to obtain
\[ 
\vartheta^*(X) = \vartheta^*(X) \cap K^{p}_\phi \subseteq \phi^*(X) \subset K^{p_0}_\phi.
\]
Using the assumption that $\vartheta^*(X)$ is dense in $K^{p_0}_\vartheta$, it follows that
\[
K^{p_0}_\vartheta = \M_{p_0} \subseteq K^{p_0}_\phi.
\]
Consequently, the quotient between $\vartheta$ and $\phi$ is a unimodular constant, hence $\M_p = K^p_\vartheta$. We have now established that $\vartheta^*(X)$ is dense in $K^p_\vartheta$, for all $1\leq p \leq \infty$. In the proceeding part of the proof, we may now without loss of generality, assume that $\vartheta^*(X)$ dense in $K^2_\vartheta$, and note that it remains to prove that $\vartheta^*(X)$ is dense in $K^p_\vartheta$, for $0<p<1$.

Recall that $\M_p$, the $H^p$-closure of $\vartheta^*(X)$ forms a backward shift invariant closed subspace on $H^p$. According to \thref{Aleks}, there exists a triple $(\phi, F, \kappa)$ consisting of an inner function $\phi$, a closed set $F$ with $\sigma(\phi) \subseteq F \subset \T$, and a non-negative integer-valued function $\kappa$ on $F$, such that $\M_p = \E^p \left( \phi, F, \kappa \right)$. From \eqref{pinv} it immediately follows that $\E^p(\phi, F, \kappa) \subseteq K_\vartheta^p$, which implies that $\phi$ divides $\vartheta$ and that functions in $\E^p(\phi, F, \kappa)$ are analytic across $\T \setminus \sigma(\vartheta)$, thus $\sigma(\phi) \subseteq F \subseteq \sigma(\vartheta)$. Note that since $H^2$ is continuously embedded in $H^p$, we have that $\mathcal{E}^p(\phi, F, \kappa)\cap H^2$ is a closed subspace in $H^2$. This observation together with the assumption that $\vartheta^*(X)$ is dense in $K_\vartheta$, now immediately gives
\[
K_\vartheta = \M_2 \subseteq \E^p(\phi, F, \kappa) \cap H^2. 
\]
However, this implies that $\vartheta$ divides $\phi$ and consequently $\sigma(\vartheta) \subseteq \sigma(\phi)$. Combining, we conclude that $\vartheta$ is a unimodular constant multiple of $\phi$, hence $F = \sigma(\vartheta)$ and 
$\E^p(\phi, F, \kappa) = \E^p(\vartheta, \sigma(\vartheta), \kappa) \subseteq K^p_\vartheta$, which at its turn implies that $\kappa \equiv 0$. This yields
\[
\M_p = \mathcal{E}^p(\phi, F, \kappa) = K_\vartheta^p,
\]
which is readily equivalent to the assertion that $\vartheta^*(X)$ is dense in $K^p_\vartheta$.
\end{proof}

The next proof in line is that of \thref{Fpropthm} and does not require any preliminary results.

\begin{proof}[Proof of \thref{Fpropthm}]
The proof consists of a simple functional analytic argument. Note that the linear manifold $X \cap K^2_\vartheta$, regarded as a subset of $X$, is easily seen to be the pre-annihilator under the Cauchy duality of the linear manifold $\vartheta \Po := \{\vartheta q: q \in \Po \}$, where $\Po$ is the set of analytic polynomials. In other words, $X \cap K^2_\vartheta = \vartheta \Po_\perp$.
As a consequence, we get that the annihilator of $X \cap K^2_\vartheta$ agrees with the weak-star closure of $\vartheta \Po$, that is

\begin{equation}\label{annihil}
\left( X \cap K^2_\vartheta \right)^\perp = \left( \vartheta \Po_\perp \right)^\perp = [ \vartheta ]_{X'}.
\end{equation}

Suppose that the pair $(X, \vartheta)$ satisfies the $(P)$-property and let $f\in K^2_\vartheta$ with $f \in \left(X \cap K^2_\vartheta\right)^{\perp}$. According to \eqref{annihil}, we have $f \in [ \vartheta]_{X'} \cap H^2 \subseteq \vartheta H^2$, which readily implies that $f\in K^2_\vartheta \cap \vartheta H^2 = \{0\}$, thus $X \cap K^2_\vartheta$ is dense in $K^2_\vartheta$. 

Conversely, suppose that $X \cap K^2_\vartheta$ is dense in $K^2_\vartheta$. For any $g \in K^2_\vartheta$, there exists a sequence $\{g_n\}_n \subset X \cap K^2_\vartheta$, such that $g_n \to g$ in $H^2$. Now given an arbitrary $f \in \left( X \cap K^2_\vartheta \right)^\perp \cap H^2$, we have 
\[
\int_{\T}f(\zeta) \conj{g(\zeta)} dm(\zeta) = \lim_{n \to \infty} \int_{\T} f(\zeta) \conj{g_n(\zeta)} dm(\zeta) =0.
\]
This shows that $f \in (K^2_\vartheta)^\perp = \vartheta H^2$. Consequently, this observation together with \eqref{annihil} implies the set inclusion
\[
[ \vartheta]_{X'} \cap H^2 = \left( X \cap K^2_\vartheta \right)^\perp \cap H^2 \subseteq \vartheta H^2.
\]
We have thus proved that the pair $(X, \vartheta)$ satisfies the $(P)$-property.



\end{proof}

In order to establish \thref{FpropHp}, we shall first introduce a lemma. The ideas behind this suggested approach appears implicitly in the work of \cite{jfabackshift}. 

\begin{lemma} \thlabel{Hlemma}
Let $p >0$ and $f \in H^p$. There exists an outer function $F: \D \to \D$ with positive real part which satisfies the following additional properties:
\begin{enumerate}[(i)]
\item $F(0) = (1+\|f\|_p)^{-1}$,
\item $|F^{1/p}(z)f(z)| \leq 1, \qquad  z \in \D$.
\end{enumerate} 
\end{lemma}

\begin{proof}
Since $f \in H^p$, it has a unique harmonic majorant $u=u_f$ on $\D$, with the properties $u(0) = \|f\|_p$ and $|f|^p \leq u$ on $\D$. Let $v$ be a harmonic conjugate of $u$ such that $v(0) = 0$ and set $G = u+1+iv$. Then on the disk we have that
 \[
  |f|^p \leq u \leq u + 1 \leq |u+ 1 + iv|  = |G|,
 \] 
which implies that $|f/G^{1/p}| \leq 1$. Set $F = G^{-1}$. Then $F$ has positive real part, $F(0) = (\|f\|_p + 1)^{-1}$, $|F| \leq 1$ on $\D$. Moreover, by the simple estimates above, we have that $|F^{1/p}f| \leq 1$ for on $\D$. Since $F$ and $1/F = G$ have positive real parts, they are both contained in $H^q$ for any $q \in (0,1)$ (see, for instance, \cite[Lemma 2.1.11]{cauchytransform}). It follows that $F$ is an outer function.
\end{proof}

With this lemma at hand, we are now ready to establish \thref{FpropHp}.

\begin{proof}[Proof of \thref{FpropHp}] Let $\vartheta$ be an arbitrary inner function and let $\N^+$ denote the Smirnov class of analytic functions on $\D$. We shall first prove that the set $X' \cap \vartheta \N^+ := \{ f \in X' : f/\vartheta \in \N^+ \}$ is weak-star closed in $X'$. Note that since $X$ is a separable Banach space and $X' \cap \vartheta \N^+$ is convex, it follows from the Krein-Smulian theorem (see \cite{dunford1988linear}) that it suffices to prove that $X' \cap \vartheta \N^+$ is the weak-star sequential closed. To this end, let $\{f_n \}_n \subset X' \cap \vartheta \N^+$ be a sequence which converges in weak-star of $X'$ to some $f \in X'$. Then by the principle of uniform boundedness, we have that $\sup_n \| f_n \|_{X'} < \infty$ . Since $X'$ is continuously embedded in $H^p$, it follows that $\sup_n \| f_n \|_{H^p} < \infty$ and that $f_n(z) \to f(z)$ for each $z \in \D$. To each function $f_n$ we may apply \thref{Hlemma} and let $F_n$ denote the corresponding outer function. Property $(i)$ of \thref{Hlemma} and the uniform boundedness of $\{f_n\}$ in $H^p$-norm imply that the values $F_n(0)$ are bounded away from $0$, and since $|F_n| \leq 1$ on $\D$, we may according to Montel's theorem and by means of passing to a subsequence, assume that the functions $F_n$ converge uniformly on compact subsets of $\D$ to a non-zero function $F$ with positive real part. The argument given in \thref{Hlemma} shows that $F$ is outer. By property $(ii)$, the functions $F^{1/p}_n f_n$ are uniformly bounded by $1$ in modulus on $\D$. Since they converge pointwise to the function $F^{1/p}f$, we conclude that the sequence $\{F^{1/p}_n f_n\}_n$ converges weakly to $F^{1/p}f$ in $H^2$. By the assumption that elements $f_n \in \vartheta \N^+$, it follows that $F^{1/p}_n f_n \in \vartheta H^2$, hence the same conclusion remains true for the weak limit: $F^{1/p}f \in \vartheta H^2$. However, since $F$ is outer, we conclude that $f/\vartheta \in \N^+$. Consequently, we have established that $X' \cap \vartheta \N^+$ is weak-star closed in $X'$. Now since $\vartheta p \in X' \cap \vartheta \N^+$ for any analytic polynomial $p$ and $X' \cap \vartheta \N^+$ is weak-star closed, we obtain the inclusion $[ \vartheta ]_{X'} \subseteq X' \cap \vartheta \N^+$. Using Smirnov's maximum principle, we get
\[
[ \vartheta ]_{X'} \cap H^2 \subseteq X' \cap \vartheta  \N^+ \cap H^2 \subseteq \vartheta H^2,
\]
thus the pair $(X ,\vartheta)$ satisfies the $(P)$-property, for all inner functions $\vartheta$.
\end{proof}

\bibliographystyle{siam}
\bibliography{mybib}

\end{document}